\documentclass[reqno, 12pt]{amsart}
\usepackage{amsmath}
 \usepackage{amssymb}
\usepackage{amsthm}
\usepackage{times}
\usepackage{latexsym}
\usepackage[mathscr]{eucal}

\numberwithin{equation}{section}
 
  \newtheorem{theorem}{Theorem}[section]

  \newtheorem{corollary}[theorem]{Corollary}

\title[Ricci solitons on a family of  three dimensional Lorentzian Walker manifolds]{Ricci solitons on a family of  three dimensional Lorentzian Walker manifolds}

\author{A. Diatta *,  M. Ciss** , A. S. Diallo***}

\newcommand{\acr}{\newline\indent}
\address{\llap{*\,} Universit\'e Assane Seck,\acr
UFR ST, D\'epartement de Math\'ematiques,\acr
Laboratoire  Math\'ematiques et Application  (LMA),\acr
 B. P. 523, Ziguinchor, S\'en\'egal}
\email{a.diatta20160578@zig.univ.sn}

\address{\llap{***\,} Universit\'e Alioune Diop,\acr
UFR SATIC, D\'epartement de Math\'ematiques,\acr
\'Equipe de Recherche en Analyse Non Lin\'eaire et G\'eom\'etrie (ANLG),\acr
 B. P. 30, Bambey, S\'en\'egal}
\email{mamadou.ciss@uadb.edu.sn}

\address{\llap{***\,} Universit\'e Alioune Diop,\acr
UFR SATIC, D\'epartement de Math\'ematiques,\acr
\'Equipe de Recherche en Analyse Non Lin\'eaire et G\'eom\'etrie (ANLG),\acr
B. P. 30, Bambey, S\'en\'egal}
\email{abdoulsalam.diallo@uadb.edu.sn}

\subjclass[2010]{53C20, 53C21}

\keywords{Einstein manifolds, Ricci soliton, Walker manifolds.}

\begin{document}

\begin{abstract}  
A Ricci soliton is a natural generalization of an Einstein metric. On a 
pseudo-Riemannian manifold $(M, g)$, it is defined by  :
$\mathcal{L}_X g + \rho = \lambda \cdot g$, where $X$ is a smooth vector field
on $M$, $\mathcal{L}_{X}$ denotes the Lie derivative in the direction of $X$, 
$\rho$ is the Ricci tensor, and $\lambda$ is a real constant. In this paper, 
we establish the existence of non-trivial Ricci solitons on a family of 
three-dimensional Lorentzian Walker manifolds.
\end{abstract}

\maketitle

\section{Introduction}

\noindent
Ricci solitons were first investigated in Riemannian and Lorentzian signature.
They have also been studied recently in a more general context by several 
authors.  Batat et al. \cite{Batat2011} show that an $n$-dimensional
Lorentzian manifold whose isometry group is of dimension at least 
$\frac{1}{2} n(n-1)+1$ admit different vector fields resulting in expanding, 
steady and shrinking Ricci solitons. Brozos-V\'azquez et al. \cite{Brozos2012}, 
analyze the existence of three-dimensional Lorentzian homogeneous Ricci 
solitons, showing the existence of shrinking, expanding and steady Ricci 
solitons in that setting.  Also, Calvaruso and De Leo \cite{CalvarusoDeLeo2011} 
give some examples of Ricci solitons within the family of three-dimensional 
Walker manifolds. Pirhadi, Fasihi-Ramandi and Azami \cite{PirhadiFasihiRamandiAzami2023} 
have generalized the Ricci soliton equation on the three-dimensional 
Lorentzian Walker manifolds.\\

A pseudo-Riemann manifold $(M, g)$ that admits a totally isotropic distribution 
field, parallel to the Levi-Civita connection, is a Walker manifold.  In dimension $3$, 
they form an important class of Lorentzian manifolds possessing a parallel isotropic vector field. Thanks to their explicit local form, they allow the study of delicate 
questions in differential geometry or physical models in relativity. 
Calvaruso \cite{Calvaruso2007} classified all homogeneous Lorentzian manifolds of dimension three. Calvaruso and Kowalski \cite{CalvarusoKowalki2009} identified 
the possible forms of the Ricci operator for locally homogeneous Lorentzian 
manifolds in dimension three.\\

Motivated, by the above works \cite{Brozos2012, CalvarusoDeLeo2011}, we study 
Ricci solitons on Walker three-manifolds $(M, g)$. For several classes of these
manifolds, described in terms of a deﬁning function $f$, existence results 
are obtained. We organize the paper as follows :  in section \ref{Description}, 
we shall describe the curvature of the metric considered. In section \ref{Soliton}, 
we study and characterize Ricci solitons via a system of partial differential 
equations involving the function $f$.

\section{Description of the metric}\label{Description}

The geometry of Walker three-manifolds  has been studied in \cite{Brozos2009}. 
These manifolds are described in terms of a suitable system of local coordinates 
$(t, x, y)$ and form a large class, depending on an arbitrary function 
$f = f (t, x, y)$. The special case of Lorentzian three-manifolds admitting 
a parallel null vector ﬁeld strictly Walker three-manifolds is characterized 
by the fact that $f = f (x, y)$. \\

We consider the Walker metric  $g_{f}$ on $\mathcal{O} \subset \mathbb{R}^3$ 
given by:
\begin{eqnarray}\label{eq2.1} 
g _f= \left( \begin{array}{cccc}
0&0&1\\
0&\varepsilon&0\\
1&0&f(t,x,y)
\end{array}
\right),
\end{eqnarray}
where $\varepsilon=\pm 1$, $f(t,x,y) = a(t,x)y^2 +b(t,x) y + d(t,x)$ and we study 
the existence of non-trivial (i.e not Einstein) Ricci solitons on Walker three-manifolds
$(M, g _f)$. A straightforward calculation shows that the non-zero components  
of the Levi-Civita connection of the metric (\ref{eq2.1}) are given by :
\begin{eqnarray}\label{eq2.2}
\nabla_{\partial_t}\partial_y &=& \frac{1}{2} f_t \partial_t, \; 
\nabla_{\partial_x}\partial_y = \frac{1}{2} f_x \partial_t, \;
 \nonumber \\
\nabla_{\partial_y}\partial_y &=& \frac{1}{2} \left( ff_t  + f_y\right)  
\partial_t 
 - \frac{1}{2\varepsilon} f_x\partial_x 
- \frac{1}{2} f_t\partial_y.
\end{eqnarray}   
Using (\ref{eq2.2}), we can completely determine the curvature tensor
of the metric (\ref{eq2.1}) by the following formula:
$\mathcal{R}(\partial_i, \partial_j)\partial_k = ([\nabla_{\partial_i}, 
\nabla_{\partial_j}]  - \nabla_{[\partial_i, \partial_j]}) \partial_k$. Then, 
taking into account (\ref{eq2.1}), we can determine all components 
of the $(0,4)$-curvature tensor : 
$R_{ijkl}=g_{a}(\mathcal{R}(\partial_i,\partial_j)\partial_k, \partial_l)$. 
We obtain that the non-zero components of the $(0,4)$-curvature tensor 
of the metric (\ref{eq2.1}) are given by :
\begin{eqnarray}\label{eq2.3}
R(\partial_x,\partial_y)\partial_x&=& -\frac{1}{2} f_{xx}\partial_t, \;
R(\partial_t,\partial_y)\partial_t =-\frac{1}{2} f_{tt}\partial_t, \nonumber \\
R(\partial_t,\partial_y)\partial_x &=& -\frac{1}{2} f_{tx}\partial_t , \;
R(\partial_x,\partial_y)\partial_t = -\frac{1}{2} f_{tx}\partial_t ,\nonumber \\
R(\partial_t,\partial_y)\partial_y&=& -\frac{1}{2} ff_{tt}\partial_t+\frac{1}{2\varepsilon} f_{tx}\partial_x+\frac{1}{2} ff_{tt}\partial_y ,\nonumber \\
R(\partial_x,\partial_y)\partial_y&=& -\frac{1}{2} ff_{tx}\partial_t+\frac{1}{2\varepsilon} f_{xx}\partial_x+\frac{1}{2} ff_{tx}\partial_y .\nonumber \\
\end{eqnarray}
By (\ref{eq2.3}), we can calculate the components $\rho_{ij}$ with respect 
to $\partial_i$ of the Ricci tensor of the metric (\ref{eq2.1}). We find that the 
non-zero components of the Ricci tensor  are given by :
\begin{eqnarray}\label{eq2.4}
\rho_{yy} = \frac{1}{2\varepsilon} \left(\varepsilon ff_{tt}-f_{xx}\right) , \;\;
\rho_{xy} = \frac{1}{2} f_{tx}, \;\;
\rho_{ty} = \frac{1}{2} f_{tt}.
\end{eqnarray}

\section{Ricci solitons on a family of three-dimensional Walker manifolds}\label{Soliton}

\noindent
In this section, we investigate the conditions for the existence of Ricci solitons 
on Walker three-manifolds $(M, g_f)$ where $g_f$ is given by (\ref{eq2.1}) . \\

A Ricci soliton is a natural generalization of an Einstein metric. It is defined 
on a pseudo-Riemannian manifold $(M, g)$ by:
\begin{eqnarray}\label{eq3.1} 
\mathcal{L}_X g + \rho = \lambda \cdot g,
\end{eqnarray}
where $X$ is a smooth vector field on $M$,  $\mathcal{L}_X$ denotes the
Lie derivative in the direction of $X$, $\rho$ is the Ricci tensor and $\lambda$
is a real number. The Ricci soliton is said to be shrinking, steady or expanding
according to whether $\lambda > 0$, $\lambda =0$, or $\lambda < 0$,
respectively. Moreover, we say that a Ricci soliton is a gradient Ricci soliton
if it admits a vector field $X$ satisfying $X = \mathrm{grad} h$ for some
potential function $h$.\\

\noindent
The description of Ricci solitons can be regarded as a ﬁrst step in understanding 
the Ricci flow, since they are the ﬁxed points of the ﬂow. Moreover, they
are important in understanding singularities of the Ricci ﬂow \cite{Brozos2012}.\\
 
Let  $X = A\partial_t+ B\partial_x + C \partial_y $ be a vector field on $(M, g_f)$, 
where $A, B, C$ are smooth functions of variables $t, x, y$. The  non-zero 
components of the Lie derivative of the Walker metric (\ref{eq2.1}) in the direction 
of $X$ are :
\begin{eqnarray}\label{eq3.2}
(\mathcal{L}_X g_f)(\partial_t,\partial_t) &=& 2C_t , \nonumber \\
(\mathcal{L}_X g_f)(\partial_t,\partial_x) &=& (\mathcal{L}_X g_f)(\partial_x,\partial_t) 
= \varepsilon B_t +C_x, \nonumber \\
(\mathcal{L}_X g_f)(\partial_t,\partial_y) &=& (\mathcal{L}_X g_f)(\partial_y,\partial_t) 
= A_t +C_y +fC_t ,  \nonumber \\
(\mathcal{L}_X g_f)(\partial_x,\partial_x) &=& 2\varepsilon B_x, \nonumber \\
(\mathcal{L}_X g_f)(\partial_x,\partial_y) &=& (\mathcal{L}_X g_f)(\partial_y,\partial_x)
= \varepsilon B_y +fC_x +A_x, \nonumber \\
(\mathcal{L}_X g_f)(\partial_y,\partial_y) &=& X(f)+2fC_y+2A_y. 
\end{eqnarray}

 Our first main result is the following:

\begin{theorem}
Let $(M, g_f)$ be a three-dimensional Walker manifold, where $g_f$ is the metric 
given by \eqref{eq2.1} and $f$ is the defining function given by 
$f(t,x,y) = a(t,x) y^{2} + b(t,x) y + d(t,x)$.  Then $(M, g_f, X, \lambda)$ is a Ricci soliton 
if equation \eqref{eq3.1} is satisfied by $X = A\partial_t+ B\partial_x + C \partial_y $ 
where $A, B $ and $C$ are given by :
\begin{eqnarray}\label{eq3.3}
A(t,x, y) &=& \left( \lambda-K_y(y)\right)t + \varepsilon H_{y}(y)
- \frac{1}{2}a_{t}(t)y^2 +N(x,y), \nonumber\\ 
B(t,x,y) &=&\frac{\lambda}{2}x + H(t,y),\nonumber\\
C(t,x,y)&=& -\varepsilon H_{t}(y)x+K(y),
\end{eqnarray}
and the functions $H, K, E$ satisfy the following conditions :
\begin{align*}
&\left( \frac{1}{2}a_{tx}(t,x)	-\varepsilon aH_{t}(t,y)\right) y^{2}
+ \left(  \frac{1}{2}b_{tx}(t,x)-\varepsilon bH_{t}(t,y)\right) y\nonumber\\
& + \frac{1}{2}d_{tx}(t,x)-\varepsilon dH_{t}(t,y)+\varepsilon H_{y}(t,y) + N_{x}(x,y) = 0,
\end{align*} 
and
\begin{align*}
&&\left(-\varepsilon H_{t}(y)x + K(y)\right) f_y + f\left( -2\varepsilon H_{ty}(y)x
+ 2K_y(y) + \lambda + \frac{1}{2}f_{tt}\right)\nonumber  \\
&& + \left( ( \lambda - tK_y(y)) + \varepsilon H_{y}(y) - \frac{1}{2}a_{t}(t)y^2 
+ N(x,y)\right) f_t\nonumber\\
&& + 4\left(\varepsilon H_{yy}(y) + N_y(x,y) - tK_{yy}(y) - a_ty \right) 
- \frac{1}{2\varepsilon}f_{xx}=0.
\end{align*}
\end{theorem}
	
\begin{proof}
Let $f(t,x,y)=a(t,x) y^{2}+b(t,x) y+d(t,x)$ and 
$X = A\partial_t+ B\partial_x + C \partial_y $ in $(M, g_f)$.  Using (\ref{eq2.1}),  
\eqref{eq3.1} and (\ref{eq3.2}), a long  calculation gives that three-dimensional 
Walker manifold $(M,g_f)$ is a Ricci soliton if and only if the following 
conditions hold :
\begin{itemize}
\item (C1) $C_t= 0$,  
\item (C2) $\varepsilon B_t +C_x= 0$,  
\item (C3) $A_t +C_y +fC_t+\frac{1}{2}f_{tt}=\lambda$,  
\item (C4) $2\varepsilon B_x=\varepsilon\lambda$,  
\item (C5) $\varepsilon B_y +fC_x +A_x+\frac{1}{2}f_{tx}=0$,  
\item (C6) $ X(f) + 2\left( fC_y + 2A_y\right) 
+ \frac{1}{2\varepsilon}\left( \varepsilon ff_{tt}-f_{xx}\right) =\lambda f$. 
\end{itemize}
The condition (C4) yields the following form : 
\begin{equation}\label{eq3.4}
B(t,x,y) =\frac{\lambda}{2}x + H(t,y),
\end{equation}
where	$H \in C^{\infty}(M)$. From the condition (C1), $C(t,x,y)=C(x,y)$. Then, using 
\eqref{eq3.4} in the condition (C2), we have :
\begin{equation}\label{eq3.5}
C(t,x,y)= -\varepsilon H_{t}(y)x+K(y),
\end{equation}  
where $K \in C^{\infty}(M)$. Let us substitute \eqref{eq3.4} and \eqref{eq3.5} 
in the condition (C3), we obtain :
\begin{equation}\label{eq3.6}
A(t,x, y) = \left( \lambda - K_y(y)\right)t + \varepsilon H_{y}(y)
- \frac{1}{2}a_{t}(t)y^2 +N(x,y),
\end{equation}
where $N \in C^{\infty}(M)$. Using \eqref{eq3.4}, \eqref{eq3.5} and \eqref{eq3.6} 
in the condition (C5), we get :
\begin{align*}
&\left( \frac{1}{2}a_{tx}(t,x)	- \varepsilon aH_{t}(t,y)\right) y^{2}
+ \left(  \frac{1}{2}b_{tx}(t,x) - \varepsilon bH_{t}(t,y)\right) y\nonumber\\
& + \frac{1}{2}d_{tx}(t,x)-\varepsilon dH_{t}(t,y)+\varepsilon H_{y}(t,y) + N_{x}(x,y)=0.
\end{align*}
By definition, we have :
\begin{eqnarray*}
X(f) &=& A(t,x,y) f_t + B(t,x,y) f_x + C(t,x,y) f_y.
\end{eqnarray*}
Thus, by the condition (C6), we have  :
\begin{align*}
A(t,x,y)f_t & + B(t,x,y)f_x + C(t,x,y) f_y\\
& + 4A_y - \frac{1}{2\varepsilon}f_{xx}  + f\left( 2C_y +\lambda
+ \frac{1}{2}f_{tt}\right)  =0.
\end{align*}
Therefore
\begin{align*}
&& \left(-\varepsilon H_{t}(y)x + K(y)\right) f_y + f\left( -2\varepsilon H_{ty}(y)x
+ 2K_y(y) +\lambda+  \frac{1}{2}f_{tt}\right)\nonumber  \\
&& + \left( ( \lambda - tK_y(y)) + \varepsilon H_{y}(y) - \frac{1}{2}a_{t}(t)y^2 
+ N(x,y)\right) f_t\nonumber\\
&& + 4\left(\varepsilon H_{yy}(y) + N_y(x,y) - tK_{yy}(y) - a_ty \right) 
-\frac{1}{2\varepsilon}f_{xx}=0.
		\end{align*}	
This complete the proof of the Theorem.
\end{proof}

We have the following consequences of the main result:

\begin{corollary}
Let $(M, g_f)$ be a Walker manifold three-dimensional where $g_f$ denotes 
the metric defined by (\ref{eq2.1}) and $f$ is the defining function given by 
$f(t,x,y)=a(t) y^{2}+b(t) y+d$. Then $(M, g_f, X, \lambda)$ is a Ricci soliton 
if equation \eqref{eq3.1} is satisfied by 
$X = A\partial_t+ B\partial_x + C \partial_y $ where $A, B $ and $C$ are given by :
\begin{eqnarray}\label{eq3.7}
A(t,x, y)&=& \left( \lambda - K_y(y)\right)t + \varepsilon H_{y}(y)
 - \frac{1}{2}a_{t}(t)y^2 + N(x,y), \nonumber\\ 
 B(t,x,y) &=& \frac{\lambda}{2}x + H(t,y),\nonumber\\
 C(t,x,y)&=& -\varepsilon H_{t}(y)x+K(y),
\end{eqnarray} 
and the functions $H, K, E$ satisfying the following conditions :
\begin{equation*}
\varepsilon H_{y}(t,y)+N_{x}(x,y)-\varepsilon fH_{t}(t,y)=0,
\end{equation*}
and
\begin{align*}
&&\left(K(y) - \varepsilon H_{t}(y)x\right) f_y + f\left(2K_y(y)
- 2\varepsilon H_{ty}(y)x +\lambda + \frac{1}{2}f_{tt}\right)\nonumber  \\
&& + \left( \left( \lambda - K_y(y)\right)t + \varepsilon H_{y}(y)
- \frac{1}{2}a_{t}(t)y^2 +N(x,y)\right) f_t\nonumber\\
&& + 4\left( N_y(x,y) + \varepsilon H_{yy}(y) - K_{yy}(y)t
- a_t(t)y \right) = 0.
\end{align*}
\end{corollary}

\begin{proof}
Let $f(t,x,y)=a(t) y^{2}+b(t) y+d$ and 
$X = A\partial_t+ B\partial_x + C \partial_y $ in $(M, g_f)$.  Using (\ref{eq2.1}),  \eqref{eq3.1} and (\ref{eq3.3}) a standard calculation gives that three-dimensional 
Walker manifold $(M,g_f)$ is a Ricci soliton if and only if the following conditions 
hold :
\begin{itemize}
\item (C'1) $C_t= 0$,  
\item (C'2) $\varepsilon B_t + C_x = 0$,  
\item (C'3) $A_t + C_y + fC_t + \frac{1}{2}f_{tt} = \lambda$,  
\item (C'4) $2\varepsilon B_x = \varepsilon\lambda$,  
\item (C'5) $ A_x +\varepsilon B_y + fC_x=0$,  
\item (C'6) $ X(f) + 2\left( fC_y + 2A_y\right) + \frac{1}{2}ff_{tt} =\lambda f$. 
\end{itemize}
Using conditions (C'1), (C'2), (C'3) and (C'4), we obtain the functions $A, B$ 
and $C$ in (\ref{eq3.7}). Applying \eqref{eq3.7} in the condition (C'5), we 
obtain :
\begin{equation*}\
\varepsilon H_{y}(t,y) + N_{x}(x,y) - \varepsilon fH_{t}(t,y)=0.
\end{equation*}
Since
\begin{eqnarray*}
X(f)&=&A(t,x,y)f_t+C(t,x,y)f_y,
\end{eqnarray*}
thus, applying $X(f)$ in the condition (C'6), we have :
 \begin{eqnarray*}
 A(t,x,y)f_t + C(t,x,y)f_y + 4A_y  + f\left( 2C_y +\lambda + \frac{1}{2}f_{tt}\right)   =0.\end{eqnarray*}
 Therefore
\begin{align*}
 && \left(K(y) - \varepsilon H_{t}(y)x\right) f_y + f\left(2K_y(y) - 2\varepsilon H_{ty}(y)x 
 + \lambda +  \frac{1}{2}f_{tt}\right)\nonumber  \\
 &&+ \left( \left( \lambda - K_y(y)\right)t + \varepsilon H_{y}(y) 
 - \frac{1}{2}a_{t}(t)y^2 + N(x,y)\right) f_t\nonumber\\
&&+4\left( N_y(x,y) + \varepsilon H_{yy}(y) - K_{yy}(y)t
- a_t(t)y \right) = 0.
 \end{align*}
 \end{proof}

 \begin{corollary}
Let $(M, g_f)$ be a flat Walker manifold three-dimensional, where $g_f$ 
denotes the metric defined by (\ref{eq2.1}) and $f$ is the defining function 
given by $f(y)=\alpha y^{2}+\beta y+\gamma$. Then$(M, g_f, X, \lambda)$ 
is a Ricci soliton  if equation \eqref{eq3.1} is satisfied by 
$X = A\partial_t+ B\partial_x + C \partial_y $ where $A, B $ and $C$ are given by :
\begin{eqnarray}\label{eq3.8}
A(t,x, y) &=& - \frac{\beta}{2} \left( \frac{\lambda}{4}y+ \delta\right)y  
		+ \frac{\lambda }{2} t +F(x,y), \nonumber\\ 
B(t,x,y) &=&\frac{\lambda}{2}x+H(t,y),\nonumber  \\ 
C(t,x,y) &=& -\varepsilon H_{t}(y)x+K(y),
\end{eqnarray}
and the functions $H,K,F$ satisfy the following conditions :
\begin{equation*}
2H_y(y) - \varepsilon fH_t(y) + F_x(x,y)=0,
\end{equation*}
and
\begin{eqnarray*}
\left(K(y) - \varepsilon H_{t}(y)x\right)f_y &+& 2\varepsilon H_{yy}(y)x - 2K_{yy}(y)t \\
	&+& 2F_{y}(x,y) + f\left( 2K_{y}(y) + \lambda\right) =0. 
\end{eqnarray*}
\end{corollary}

\begin{proof}
Let $f(y)=\alpha y^{2}+\beta y+\gamma$ and 
$X = A\partial_t+ B\partial_x + C \partial_y $ in $(M, g_f)$.  Using (\ref{eq2.1}),  \eqref{eq3.1} and (\ref{eq3.3}), a standard calculation gives that three-dimensional 
Walker manifold $(M,g_f)$ is a Ricci soliton if and only if the following conditions hold :
\begin{itemize}
\item (C''1) $C_t= 0$,  
\item (C''2) $\varepsilon B_t +C_x= 0$,  
\item (C''3) $A_t +C_y +fC_t=\lambda$,  
\item (C''4) $2\varepsilon B_x=\varepsilon\lambda$,  
\item (C''5) $\varepsilon B_y +fC_x +A_x=0$,  
\item (C''6) $ X(f)+2\left( fC_y+2A_y\right)  =\lambda f$. 
\end{itemize}
From condition (C''1), we have: $C(t,x,y) = C(x,y)$. We establish the form of $B$ 
under condition (C"4) :
\begin{equation}\label{eq3.9}
B(t,x,y)= \frac{\lambda}{2}x + H(t,y),
\end{equation}
with $H \in C^{\infty}(M)$. We compute $B_t$ and substitute it into condition (C"2), 
we obtain : $\varepsilon H_{t}(t,y) + C_{x}=0$. Therefore
\begin{equation}\label{eq3.10}
	C(t,x,y)=-\varepsilon H_{t}(y)x + K(y),
\end{equation}
where $K\in C^{\infty}$. Using \eqref{eq3.10} in condition (C"3), we have :
\begin{eqnarray*}
A_t -\varepsilon H_{ty}(y)x + K_{y}(y) = \lambda,
\end{eqnarray*} 
and then
\begin{equation}\label{eq3.11}
A(t,x,y)=\varepsilon H_{y}(y)x+\left( \lambda-K_{y}(y)\right) t + F(x,y),
\end{equation}
where $F\in C^{\infty}(M)$. \\

We differentiate conditions (C"3) and (C"5) respectively with respect to $t$ and $x$, 
we get :
\begin{eqnarray*}
A_{tx} + C_{yx}=0 \quad \mbox{and}\quad
A_{xt} + f_{t}C_{x}+\varepsilon B_{yt} = 0.
\end{eqnarray*}
Hence
\begin{eqnarray*}
- C_{yx} + f_tC_x + \varepsilon B_{yt} = 0.
\end{eqnarray*}
Using \eqref{eq3.9} and \eqref{eq3.10}, we have : $ H_{yt}(y)=0$ and
so $H_{y}(y,t) = H_{y}(y)$. Finally we have :
\begin{eqnarray*}
A(t,x,y) &=& \varepsilon H_{y}(y)x + \left( \lambda - K_{y}(y)\right) t + F(x,y),\\
B(t,x,y) &=& \frac{\lambda}{2}x + H(t,y),\\
C(t,x,y) &=& - \varepsilon H_{t}(y)x+K(y).
\end{eqnarray*}
Thus, condition $(C"5)$, gives us
\begin{equation*}
2H_y(y) - \varepsilon f H_t(y) + F_x(x,y)=0.
\end{equation*}
Using the condition (C"6), we obtain :
\begin{eqnarray*}
X(f) + 2\varepsilon H_{yy}(y)x - 2K_{yy}(y)t + 2F_{y}(x,y) 
+ f\left( 2K_{y}(y)+\lambda\right) &=& 0.
\end{eqnarray*}
Compute $X(f)$, we have :
\begin{eqnarray*}
X(f) = \left[ -\varepsilon H_{t}(y)x + K(y)\right] f_y. 
\end{eqnarray*}
Then
\begin{eqnarray*}
\left( -\varepsilon H_{t}(y)x + K(y)\right)f_y &+& 2\varepsilon H_{yy}(y)x 
-2K_{yy}(y)t \\
&+& 2F_{y}(x,y) + f\left( 2K_{y}(y) + \lambda\right) =0.   
	\end{eqnarray*}
\end{proof}

We have the following result:
\begin{corollary}
Let $(M, g_f)$ is strict Walker manifold three-dimensional, where $g_f$ is the metric given by \eqref{eq2.1} and $f$ is the defining function given by 
$f(x,y)=a(x) y^{2}+b(x) y+d(x)$.  Then $(M, g_f, X, \lambda)$ is a Ricci soliton if 
equation \eqref{eq3.1} is satisfied by $X = A\partial_t+ B\partial_x + C \partial_y $ 
with $A, B $ and $C$ are given by
\begin{eqnarray}\label{eq3.12}
A(t,x, y) &=& \left( \lambda - K_y(y)\right)t+\varepsilon H_{y}(y) +N(x,y), \nonumber\\ 
B(t,x,y) &=&\frac{\lambda}{2}x + H(t,y),\nonumber\\
C(t,x,y) &=& -\varepsilon H_{t}(y)x+K(y),
\end{eqnarray}
where $H,K,N$ satisfy the following conditions
\begin{align*}
&	- \varepsilon a(x)H_{t}(t,y) y^{2} - \varepsilon b(x)H_{t}(t,y) y\\
&-\varepsilon d(x)H_{t}(t,y)+\varepsilon H_{y}(t,y)+N_{x}(x,y)=0,
\end{align*}
 and
 \begin{align*}
&&\left(-\varepsilon H_{t}(y)x+K(y)\right) f_y+f\left( 2K_y(y) +\lambda
-2\varepsilon H_{ty}(y)x\right)  \\
 && + 4\left( \varepsilon H_{yy}(y) + N_y(x,y)-K_{yy}(y)t\right) 
 - \frac{1}{2\varepsilon}f_{xx}=0.
 \end{align*}	
\end{corollary}

\section*{Acknowledgments}
The authors would like to thank the referee for his/her valuable suggestions and comments that helped them improve the paper. This paper was completed at a
time when the authors was visiting the African Institute for Mathematical Sciences
(AIMS-Senegal) at Mbour.


\begin{thebibliography}{10}
\addcontentsline{toc}{chapter}{Bibliographie}
\bibitem{Batat2011}
W. Batat, M. Brozos-V\'azquez, E. Garc\'ia-R\'io and S. Gavino-Fern\'andez, 
\emph{Ricci solitons on Lorentzian manifolds with large isometry groups},
Bull. Lond. Math. Soc. {\bf 43} (2011), (6), 1219-1227.

\bibitem{Brozos2012}
M. Brozos-V\'azquez,  G Calvaruso, E. Garc\'ia-R\'io and  S. Gavino-Fern\'andez, 
\emph{Three-dimensional Lorentzian homogeneous Ricci solitons}, 
Israel J. Math. {\bf 188} (2012), 385-403.

\bibitem{Brozos2009}
M. Brozos-V\'azquez, E. Garcı\'ia-Rio, P. Gilkey, S. Nikevi\'c and 
R. V\'azquez-Lorenzo. The Geometry of Walker Manifolds. Synthesis Lectures on Mathematics and Statistics, 5, (2009). (Morgan and Claypool Publishers, Williston, VT).

\bibitem{CalvarusoDeLeo2011}
G. Calvaruso and B. De Leo, 
\emph{Ricci solitons on Lorentzian Walker three-manifolds}, 
Acta Math. Hungar. {\bf 132} (2011), (3), 269-293.

\bibitem{Calvaruso2007}
G. Calvaruso, 
\emph{Homogeneous structures on three-dimensional Lorentzian manifolds}, Journal of Geometry and Physics 57 (2007), 1279–1291.

\bibitem{CalvarusoKowalki2009}
 G. Calvaruso and O. Kowalski, 
\emph{On the Ricci operator of locally homogeneous Lorentzian
$3$-manifolds},
 Cent. Eur. J. Math. {\bf 7} (2009), (1), 124-139 . 


\bibitem{PirhadiFasihiRamandiAzami2023}
V. Pirhadi, G. Fasihi‑Ramandi and S. Azami
\emph{Generalized Ricci solitons on three‑dimensional Lorentzian
Walker manifolds}, 
J. Nonlinear Math. Phys. {\bf 30},  (2023), (4),1409-1423. 



\end{thebibliography}
\end{document}